\documentclass{amsart}

\usepackage[T1]{fontenc}
\usepackage[a4paper]{geometry}
\usepackage{amsmath, amssymb, mathtools, tikz-cd, amsthm, csquotes}
\usepackage{quiver}

\usepackage[hidelinks]{hyperref}

\title{Groupoidal 2-quasi-categories and homotopy 2-types}
\author{Victor Brittes}
\date{}

\newtheorem{theorem}{Theorem} [section]
\newtheorem{proposition}[theorem]{Proposition} 
\newtheorem{corollary}[theorem]{Corollary} %[theorem]
\newtheorem{lemma}[theorem]{Lemma} %
\newtheorem{mainthm}{Theorem}

\theoremstyle{definition}
\newtheorem{parag}[theorem]{}
\newtheorem{definition}[theorem]{Definition}

\theoremstyle{remark}
\newtheorem{example}[theorem]{Example}
\newtheorem{remark}[theorem]{Remark}

\def\a{\mathcal{A}}
\def\b{\mathcal{B}}
\def\cc{\mathcal{C}}
\def\d{\mathcal{D}}
\def\g{\mathcal{G}}
\def\m{\mathcal{M}}
\def\n{\mathcal{N}}

\def\R{\mathbb{R}}
\def\L{\mathbb{L}}
\def\W{\mathsf{W}}
\def\Fib{\mathsf{Fib}}
\def\Cof{\mathsf{Cof}}

\DeclareMathOperator{\op}{op}

\DeclareMathOperator{\id}{id}

\DeclareMathOperator{\st}{st}
\DeclareMathOperator{\Hom}{Hom}
\DeclareMathOperator{\ho}{ho}
\DeclareMathOperator{\Ho}{Ho}
\DeclareMathOperator{\loc}{loc}
\DeclareMathOperator{\Set}{\mathbf{Set}}
\DeclareMathOperator{\Cat}{\mathbf{Cat}}

\DeclareMathOperator{\Bicat}{\mathbf{Bicat}}
\DeclareMathOperator{\twocat}{\mathbf{2-Cat}}
\DeclareMathOperator{\twoqcat}{\mathbf{2-QCat}}
\DeclareMathOperator{\2Gpd}{\mathbf{2-Gpd}}



\date{28 November 2022}

\begin{document}
\maketitle

\begin{abstract}
%We define a notion of groupoidal 2-quasi-category as a 2-quasi-category enriched over Kan complexes and whose underlying quasi-category is also a Kan complex. We show that there is a model category structure on the category of $\Theta_2$-sets whose fibrant objects are groupoidal 2-quasi-categories, and that this model structure is Quillen equivalent to the Kan-Quillen category structure on simplicial sets. Moreover, we show that 2-truncated groupoidal 2-quasi-categories are models for homotopy 2-types. 

We define a notion of groupoidal 2-quasi-categories and show that they are the fibrant objects of a model structure on the category of $\Theta_2$-sets. We show that this model category is Quillen equivalent to the Kan-Quillen model category of simplicial sets and that 2-truncated groupoidal 2-quasi-categories are models for homotopy 2-types.
\end{abstract}

\tableofcontents

\section*{Introduction}

One of the motivations for the development of higher category theory comes from algebraic topology. To fully encode the homotopy data of a given space, one is led to consider mathematical structures formed by objects, morphisms, 2-morphisms between morphisms, 3-morphisms between 2-morphisms... and where notions such as associativity and invertibility of morphisms are only defined up to invertible higher morphisms. In this philosophy, an $(\infty, n)$-category would be a higher category where all morphisms of dimension $> n$ are invertible. Many models have been proposed for $(\infty, n)$-categories, and comparisons of such models can be found in \cite{bergner2013comparison} and \cite{bergner2020comparison}. An object modeling $(\infty, n)$-categories is groupoidal if all $k$-morphisms $(1 \leq k \leq n)$ are also (weakly) invertible, in an appropriate sense. By considering $n$-truncated versions of these groupoidal objects -- where, roughly speaking, the data of $k$-morphisms for $k > n$ is trivial -- we expect to obtain a model for homotopy $n$-types.

For example, in the case $n = 1$, one of the several models for $(\infty, 1)$-categories are quasi-categories, which were introduced by Boardman and Vogt in the 70's and whose theory was further developed by mathematicians such as Joyal and Lurie. Quasi-categories are simplicial sets satisfying a certain lifting condition, and can also be described as fibrant-cofibrant objects of a model structure on the category of simplicial sets. An important result of Joyal is that quasi-categories where all morphisms are invertible -- which we may call groupoidal -- are precisely Kan complexes. With the notion of truncation for quasi-categories presented in \cite{joyal2008notes, campbell2020truncated}, 1-truncated groupoidal quasi-categories are exactly homotopy 1-types.

The aim of this short paper is to study similar notions in the case of 2-quasi-categories. In \cite{ara2014higher}, Ara introduces $n$-quasi-categories as models for $(\infty, n)$-categories. These are the fibrant-cofibrant objects of a model category structure on the category of $n$-cellular sets (i.e., functors $\Theta_n^{\op} \to \Set$). The category $\Theta_n$ plays the role of a generalisation of the simplex category $\Delta$, allowing to define notions of non-invertible higher morphisms.

Given a 2-quasi-category, one can consider its underlying quasi-category and quasi-categories of morphisms between any two objects, as described in \cite{campbell2020homotopy}. We propose a definition of groupoidal 2-quasi-categories as 2-quasi-categories which are locally Kan complexes (or $\infty$-groupoids), and whose underlying quasi-category is a Kan complex. We shall call these objects \textit{2-quasi-groupoids}, for short. By describing 2-quasi-groupoids as local objects in Ara's model structure for 2-quasi-categories, we obtain, starting from a theorem of \cite{campbell2020homotopy}, the following result (Theorem \ref{Quillen equivalence (infty, 2)-groupoids and spaces}):

\begin{mainthm} \label{Theorem (infty, 2)-groupoids are spaces}
There is a Quillen equivalence between the Kan-Quillen model structure on simplicial sets and a model structure on the category of 2-cellular sets whose fibrant-cofibrant objects are 2-quasi-groupoids.
\end{mainthm}

Then, we consider 2-truncated 2-quasi-groupoids (i.e., 2-quasi-groupoids which are 2-truncated 2-quasi-categories in the sense of \cite{campbell2020homotopy}). Using, once again, the theory of localisation of model category structures, we show (Theorem \ref{Quillen equivalence 2-truncated (infty, 2)-groupoids and 2-groupoids}) that

\begin{mainthm} \label{Theorem 2-truncated (infty, 2)-groupoids are 2-groupoids}
There is a Quillen equivalence between a model structure on the category of 2-cellular sets whose fibrant-cofibrant objects are 2-truncated 2-quasi-groupoids and Moerdijk-Svensson's model structure on the category of (strict) 2-groupoids, described in \cite{moerdijk1993algebraic}.
\end{mainthm}

The link to homotopy 2-types is made using a result of \cite{moerdijk1993algebraic} which provides a Quillen equivalence between 2-groupoids and homotopy 2-types.

We may summarize the discussion above by placing each model in the framework of (weak) $(r,n)$-categories: higher categories where all $k$-morphisms are invertible for $k > n$ and trivial for $k > r$. Models for general (weak) $(r,n)$-categories can be found in \cite{rezk2010cartesian}.

\vspace{2mm}

\begin{center}
\begin{tabular}{|c | c |} 
 \hline
 General concept & Model \\ 
 \hline\hline
 $(\infty, 2)$-categories & 2-quasi-categories  \\ 
 \hline
 $(\infty, 0)$-categories = $\infty$-groupoids = spaces & 2-quasi-groupoids  \\
 \hline
 $(2,2)$-categories = weak 2-categories & 2-truncated 2-quasi-categories \\
 \hline
 $(2,0)$-categories = weak 2-groupoids & 2-truncated 2-quasi-groupoids  \\
 \hline
\end{tabular}
\end{center}

\vspace{2mm}

It should be possible to generalize to $n \geq 3$ the comparison between homotopy $n$-types and $n$-truncated groupoidal $n$-quasi-categories, for a suitable notion of such. However, other methods are necessary, since it is known that strict 3-groupoids do not model all homotopy 3-types, for example (cf. \cite[2.7]{simpson2011homotopy}). This generalization is the object of future work of the author.

\vspace{2mm}

\textbf{Organisation of the paper.} In the first preliminary section, we recall some aspects of the theory of localisation of model structures, our main technical tool. In section 2, we briefly recall the definition of the category $\Theta_2$ and of 2-quasi-categories. The notion of 2-quasi-groupoid is presented in the third section, where it is also proved that 2-quasi-groupoids are models for spaces (Theorem \ref{Theorem (infty, 2)-groupoids are spaces}). Section 4 is a recollection of some results of \cite{campbell2020homotopy} concerning the homotopy coherent nerve and 2-truncated 2-quasi-categories, which will allow us to prove the equivalence between 2-truncated 2-quasi-groupoids and 2-groupoids (Theorem \ref{Theorem 2-truncated (infty, 2)-groupoids are 2-groupoids}) in section 5. The last section presents Moerdijk-Svensson's comparison between 2-groupoids and homotopy 2-types.

\vspace{2mm}

\textbf{Acknowledgement.} We would like to thank Muriel Livernet and Clemens Berger for the helpful and insightful conversations. This work is part of my Master's thesis written under their supervision. We would also like to highlight the importance of Alexander Campbell's article \cite{campbell2020homotopy} for the ideas involved in the proofs of the main theorems.

\vspace{2mm}

\textbf{Notation.} If $\a$ is a small category, we denote by $\widehat{\a}$ the category of presheaves of sets on $\a$, i.e., of functors $\a^{\op} \to \Set$. If $X$ is an object of $\widehat{\a}$ and $a$ is an object of $\a$, we write $X_a$ for the set $X(a)$. The Yoneda embedding $a \mapsto \a[a] := \Hom_{\a}(-, a)$ defines a fully faithful functor $\a \to \widehat{\a}$. If $f: a \to b$ is a morphism in $\a$, we still denote by $f: \a[a] \to \a[b]$ its image under the Yoneda embedding.

When representing an adjunction by $F: \cc \rightleftarrows \d: G$, the functor $F$ is left adjoint to the functor $G$.

\section{Localisation of model category structures}

We recall some notions and results about the localisation of model categories, mainly following the appendix A of \cite{campbell2020truncated}. A complete reference is \cite{hirschhorn2003model}.

\begin{parag} \label{bousfield localisation}
Let $(\m, \Cof, \W, \Fib)$ be a model category structure on a category $\m$. A model category structure $(\m, \Cof_{\loc}, \W_{\loc}, \Fib_{\loc})$ on $\m$ is a \textit{(left) Bousfield localisation} of $(\m, \Cof, \W, \Fib)$ if $\Cof_{\loc} = \Cof$ and $\W \subset \W_{\loc}$. When studying a model category structure and a given localisation, we shall write $\m$ and $\m_{\loc}$ to refer to the original model structure and to its localisation, respectively. We shall also call \textit{local fibration} (resp. \textit{local fibrant object}, resp. \textit{local weak equivalence}) a fibration (resp. fibrant object, resp. weak equivalence) of $\m_{\loc}$. We see that a Bousfield localisation is completely determined by its fibrant objects, i.e., the local fibrant objects. It is useful to know that a morphism between local fibrant objects is a weak equivalence (resp. fibration) in $\m$ if and only if it is a local weak equivalence (resp. local fibration).
\end{parag}

\begin{parag} \label{homotopy reflection}
A Quillen adjunction $F: \m \rightleftarrows \n: G$ is said to be a \textit{homotopy reflection} if the right-derived functor $\R G: \Ho(\n) \to \Ho(\m)$ is fully faithful. For example, if $\m_{\loc}$ is a Bousfield localisation of $\m$, then the adjunction $\id_\m: \m \rightleftarrows \m_{\loc}: \id_\m$ is a homotopy reflection (see \cite[Proposition 7.19]{joyal2007quasi}). Homotopy reflections were originally introduced in \cite{dugger2001combinatorial}, where they were called homotopically surjective maps of model categories.
\end{parag}

There is a criterion to know whether a homotopy reflection $F: \m \rightleftarrows \n: G$ yields a Quillen equivalence after localisation of $\m$.

\begin{theorem} \label{criterion homotopy refection Quillen equivalence}
Let $F: \m \rightleftarrows \n: G$ be a Quillen adjunction between model categories and $\m_{\loc}$ be a Bousfield localisation of $\m$. 

(1) The adjunction $F: \m_{\loc} \rightleftarrows \n: G$ is a Quillen adjunction if and only $G$ sends every fibrant object of $\n$ to a fibrant object of $\m_{\loc}$.

(2) Suppose that $F: \m \rightleftarrows \n: G$ is a homotopy reflection. The adjunction $F: \m_{\loc} \rightleftarrows \n: G$ is a Quillen equivalence if and only if it is a Quillen adjunction and for every fibrant-cofibrant object $X$ of $\m_{\loc}$, there is a fibrant object $Y$ of $\n$ and a weak equivalence $X \to G(Y)$ in $\m$.
\end{theorem}

\begin{proof}
(1) The condition is necessary since right Quillen functors preserve fibrant objects. To prove that it is sufficient, it is enough to show that if $G$ sends fibrant objects of $\n$ to fibrant objects of $\m_{\loc}$, then $F: \m_{\loc} \to \n$ preserves acyclic cofibrations.

Let $i$ be an acyclic cofibration of $\m_{\loc}$. We want to show that the cofibration $F(i)$ is acyclic. By \cite[Lemma 7.14]{joyal2007quasi}, it is the case if and only if $F(i)$ has the left lifting property with respect to all fibrations between fibrant objects. Let $p: X \to Y$ be such a fibration. By adjunction, we know that $F(i)$ has the left lifting property with respect to $p$ if and only if $i$ has the left lifting property with respect to $G(p)$. But $G(p)$ is a fibration of $\m$ (since $G: \n \to \m$ is right Quillen) between local fibrant objects (by the hypothesis), thus a fibration of $\m_{\loc}$. Therefore, the lift exists since $i$ is an acyclic fibration of $\m_{\loc}$.

(2) See \cite[Theorem A.14]{campbell2020truncated}.
\end{proof}

\begin{parag} \label{homotopy mapping space and local objects}
Given a model category $\m$ and two objects $X, Y$ of $\m$, we can consider the \textit{homotopy mapping space} $\underline{\Ho \m}(X, Y)$, which is the image of the pair $(X, Y)$ by the functor $\underline{\Ho \m}: \Ho(\m)^{\op} \times \Ho(\m) \to \Ho(\widehat{\Delta})$\footnote{When writing $\Ho(\widehat{\Delta})$, we always consider the Kan-Quillen model structure on the category of simplicial sets} induced by $\Hom_{\m}: \m^{\op} \times \m \to \Set$ (cf. \cite[A.2]{ara2014higher}). 

Let $S$ be a class of morphisms of $\m$. An object $X$ of $\m$ is \textit{$S$-local} (or \textit{local with respect to $S$}) if for every morphism $f: A \to B$ of $S$, the induced map
$$\underline{\Ho \m}(f, X): \underline{\Ho \m}(B, X) \to \underline{\Ho \m}(A, X)$$
is an isomorphism (in the homotopy category $\Ho(\widehat{\Delta})$).

A morphism $f: A \to B$ of $\m$ is an \textit{$S$-equivalence} if for every $S$-local object $X$, the induced map
$$\underline{\Ho \m}(f, X): \underline{\Ho \m}(B, X) \to \underline{\Ho \m}(A, X)$$
is an isomorphism (in the homotopy category $\Ho(\widehat{\Delta})$).

If there is a Bousfield localisation $\m_{\loc}$ of $\m$ whose local fibrant objects are the $S$-local objects and whose weak equivalences are the $S$-equivalences, we say that $\m_{\loc}$ is a \textit{(Bousfield) localisation of $\m$ with respect to $S$}, and we denote it by $L_S \m$.
\end{parag}

\begin{example} \label{Kan-Quillen vs Joyal}
The Kan-Quillen model category structure on simplicial sets is a localisation of Joyal's model structure with respect to the morphism $\Delta[1] \to \Delta[0]$, cf. \cite[Proposition 3.30]{campbell2020truncated}.
\end{example}

We will state a result, due to Smith, about the existence of the localisation of a model category with respect to a certain set of morphisms. Before, let us recall some definitions. A model category is \textit{left proper} if the pushout of every weak equivalence along a cofibration is a weak equivalence. A model category where all objects are cofibrant is left proper (see \cite[Corollary 13.1.3]{hirschhorn2003model}). A model category is \textit{combinatorial} if it is cofibrantly generated and locally presentable. 

%\begin{example}
%If $\a$ is a small category, all model structures defined on $\widehat{\a}$ via \ref{thm model structure ex nihilo} are left proper and combinatorial.
%\end{example}

\begin{theorem} \label{smith existence theorem}
Let $\m$ be a left proper and combinatorial model category. Let $S$ be a set of morphisms of $\m$. Then the localisation $L_S \m$ of $\m$ with respect to $S$ exists and is left proper and combinatorial.
\end{theorem}

\begin{proof}
See \cite[Theorem 4.7]{barwick2010left}.
\end{proof}

\begin{remark}
If $F: \m \rightleftarrows \n: G$ is a Quillen adjunction between model categories, the induced adjunction between the homotopy categories is usually denoted by $\L F: \Ho (\m) \rightleftarrows \Ho(\n): \R G$. In what follows, we will abuse language and also denote by $\L F$ the functor $\L F := F Q: \m \to \n$, where $Q$ is a fixed functorial cofibrant replacement in the model category $\m$. In all the applications presented in this paper, all objects of $\m$ will be cofibrant, and so we shall take $Q = \id_{\m}$.
\end{remark}

We can transfer localisations of model structures along Quillen adjunctions.

\begin{proposition} \label{local fibrant objects and localisation}
Let $F: \m \rightleftarrows \n: G$ be a Quillen adjunction between model categories $\m$ and $\n$. Let $S$ be a class of morphisms of $\m$. A fibrant object $Y$ of $\n$ is $\L F(S)$-local if and only if $G(Y)$ is $S$-local. 
\end{proposition}

\begin{proof}
See \cite[Proposition 3.1.12]{hirschhorn2003model}.
\end{proof}

\begin{theorem} \label{transfer of localisation}
Let $F: \m \rightleftarrows \n: G$ be a Quillen adjunction between model categories $\m$ and $\n$. Let $S$ be a class of morphisms of $\m$. If the localisations $L_S \m$ and $L_{\L F(S)} \n$ exist, then 
$$F: L_S \m \rightleftarrows L_{\L F(S)} \n: G$$
is a Quillen adjunction between the localised model categories. 

Moreover, if $F: \m \rightleftarrows \n: G$ is a Quillen equivalence, then so is $F: L_S \m \rightleftarrows L_{\L F(S)} \n: G$.
\end{theorem}

\begin{proof}
See \cite[Proposition 3.3.20]{hirschhorn2003model}
\end{proof}

We end this section stating a proposition which will allow us to understand successive localisations of a model category.

\begin{lemma} \label{lemma successive localisations}
Let $\m$ be a model category and $\m_{\loc}$ be a Bousfield localisation of $\m$. For every object $X$ and every local fibrant object $Y$ of $\m$, the homotopy mapping spaces $\underline{\Ho \m}(X, Y)$ and $\underline{\Ho \m_{\loc}}(X, Y)$ are naturally isomorphic in $\Ho(\widehat{\Delta})$.
\end{lemma}

\begin{proof}
See \cite[Lemma A.4]{ara2014higher}
\end{proof}

\begin{proposition} \label{proposition succesive localisations}
Let $\m$ be a model category and $S, T$ be two classes of morphisms of $\m$. Suppose that the localisations $L_S \m$, $L_T \m$, $L_T L_S \m$, $L_S L_T \m$ and $L_{S \cup T} \m$ exist. Then the model categories $L_T L_S \m$, $L_S L_T \m$ and $L_{S \cup T} \m$ are the same.
\end{proposition}

\begin{proof}
Since a model structure is completely determined by its cofibrations and fibrant objects (cf. \cite[Proposition E.1.10]{joyal2008thetheory}) and the 3 considered model structures have the same cofibrations, it is sufficient to show that they have the same fibrant objects. Let $X$ be an object of $\m$. We claim that the following assertions are equivalent:
\begin{enumerate}
    \item $X$ is a $T$-local object of $L_S \m$
    \item $X$ is an $S$-local object of $L_T \m$
    \item $X$ is an $(S \cup T)$-local objects of $\m$
\end{enumerate}

We will show that $(1) \Leftrightarrow (3)$. The equivalence $(2) \Leftrightarrow (3)$ follows by exchanging the roles of $S$ and $T$.

$(1) \Rightarrow (3)$ Suppose that $X$ is a $T$-local object of $L_S \m$. We have to show that, for every $f \in S \cup T$, $f: A \to B$, the map
$$\underline{\Ho \m}(f, X): \underline{\Ho \m}(B, X) \to \underline{\Ho \m}(A, X)$$
is an isomorphism. This is true if $f \in S$, since $X$ is fibrant in $L_T L_S \m$, so it is in particular fibrant in $L_S \m$, which means it is $S$-local in $\m$. If $f \in T$, we consider the commutative square
% https://q.uiver.app/?q=WzAsNCxbMCwwLCJcXHVuZGVybGluZXtcXHRleHR7SG99IFxcbWF0aGNhbHtNfX0oQiwgWCkiXSxbMiwwLCJcXHVuZGVybGluZXtcXHRleHR7SG99IFxcbWF0aGNhbHtNfX0oQSwgWCkiXSxbMCwxLCJcXHVuZGVybGluZXtcXHRleHR7SG99IExfU1xcbWF0aGNhbHtNfX0oQiwgWCkiXSxbMiwxLCJcXHVuZGVybGluZXtcXHRleHR7SG99IExfU1xcbWF0aGNhbHtNfX0oQSwgWCkiXSxbMCwxLCJcXHVuZGVybGluZXtcXHRleHR7SG99IFxcbWF0aGNhbHtNfX0oZiwgWCkiXSxbMiwzLCJcXHVuZGVybGluZXtcXHRleHR7SG99IExfU1xcbWF0aGNhbHtNfX0oZiwgWCkiXSxbMCwyLCJcXGNvbmciLDJdLFsxLDMsIlxcY29uZyJdXQ==
\[\begin{tikzcd}
	{\underline{\text{Ho} \mathcal{M}}(B, X)} && {\underline{\text{Ho} \mathcal{M}}(A, X)} \\
	{\underline{\text{Ho} L_S\mathcal{M}}(B, X)} && {\underline{\text{Ho} L_S\mathcal{M}}(A, X)}
	\arrow["{\underline{\text{Ho} \mathcal{M}}(f, X)}", from=1-1, to=1-3]
	\arrow["{\underline{\text{Ho} L_S\mathcal{M}}(f, X)}", from=2-1, to=2-3]
	\arrow["\cong"', from=1-1, to=2-1]
	\arrow["\cong", from=1-3, to=2-3]
\end{tikzcd}\]
where the isomorphisms are given by Lemma \ref{lemma successive localisations}. The bottom arrow is an isomorphism, since $X$ is $T$-local in $L_S \m$ by assumption, and thus the top arrow is also an isomorphism, as desired.

$(3) \Rightarrow (1)$ Let $X$ be a $(S \cup T)$-local object of $\m$. Let $f: A \to B$ be a morphism in the class $T$. We have to show that the map
$$\underline{\Ho L_S \m}(f, X): \underline{\Ho L_S \m}(B, X) \to \underline{\Ho L_S \m}(A, X)$$
is an isomorphism. Since $X$ is $(S \cup T)$-local in $\m$, it is in particular $S$-local in $\m$, and so it is a fibrant object of $L_S \m$. Therefore, we can use Lemma \ref{lemma successive localisations} to consider a commutative square as above. The top arrow is an isomorphism since $X$ is $T$-local in $\m$, which implies that the bottom arrow is also an isomorphism.
\end{proof}

\section{2-quasi-categories}

Ara introduces in \cite{ara2014higher} $n$-quasi-categories as models for $(\infty, n)$-categories. These are presheaves on the category $\Theta_n$ which are fibrant objects for a certain model category structure on $\widehat{\Theta_n}$. Here, we consider the case $n=2$.

\begin{parag} \label{the category Delta}
We denote by $\Delta$ the category of finite ordinals $[n] = \{0 < 1 < \ldots < n\}$ and non-decreasing maps between them. There is a fully faithful inclusion $\Delta \to \Cat$, where we see each object of $\Delta$ as the category associated to the respective ordered set. 

The morphisms of $\Delta$ are generated by the face and degeneracy maps. Face maps are the maps $\delta^{i}: [n] \to [n+1]$ skipping $i$, for $0 \leq i \leq n+1$. Degeneracy maps are the maps $\sigma^i: [n+1] \to [n]$ sending both $i$ and $i+1$ to $i$, for $0 \leq i \leq n$.
\end{parag}

\begin{parag} \label{the category Theta_2}
We introduce the category $\Theta_2$ as a wreath product $\Delta \wr \Delta$, as first presented by Berger in \cite{berger2007iterated}. The objects of $\Theta_2$ are lists $[n; \mathbf{q}] = [n; q_1, \ldots, q_n]$, where $n, q_1, \ldots, q_n$ are non-negative integers. A morphism $[\alpha, \boldsymbol{\alpha}] : [m; \mathbf{p}] \to [n; \mathbf{q}]$ is the data of a morphism $\alpha: [m] \to [n]$ in $\Delta$ and of a morphism $\alpha_j: [p_i] \to [q_j]$ in $\Delta$ for every $\alpha(i-1) < j \leq \alpha(i)$, $1 \leq i \leq m$.

We shall think of $\Theta_2$ as a full subcategory of $\twocat$. An object $[n; \mathbf{q}]$ of $\Theta_2$ is represented by the 2-category freely generated by the 2-graph with $n+1$ objects $0, 1, \ldots, n$, an ordered set $\{(j,0), \ldots, (j, q_j)\}$ of  1-arrows $j-1 \to j$ for every $1 \leq j \leq n$, and one 2-arrow between any two consecutive 1-arrows between the same objects. For example, the 2-graph which generates the object $[3; 1, 0, 2]$ can be represented as below.

% https://q.uiver.app/?q=WzAsNCxbMCwwLCIwIl0sWzMsMCwiMSJdLFs2LDAsIjIiXSxbOSwwLCIzIl0sWzAsMSwiKDEsMCkiLDEseyJjdXJ2ZSI6LTN9XSxbMCwxLCIoMSwxKSIsMix7ImN1cnZlIjozfV0sWzEsMiwiKDIsMCkiLDFdLFsyLDMsIigzLDApIiwxLHsiY3VydmUiOi01fV0sWzIsMywiKDMsMikiLDEseyJjdXJ2ZSI6NX1dLFsyLDMsIigzLDEpIiwxXSxbNCw1LCIiLDEseyJzaG9ydGVuIjp7InNvdXJjZSI6MjAsInRhcmdldCI6MjB9fV0sWzcsOSwiIiwxLHsic2hvcnRlbiI6eyJzb3VyY2UiOjIwLCJ0YXJnZXQiOjIwfX1dLFs5LDgsIiIsMSx7InNob3J0ZW4iOnsic291cmNlIjoyMCwidGFyZ2V0IjoyMH19XV0=
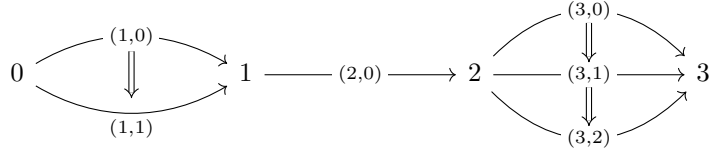
\[\begin{tikzcd}
	0 &&& 1 &&& 2 &&& 3
	\arrow[""{name=0, anchor=center, inner sep=0}, "{(1,0)}"{description}, curve={height=-18pt}, from=1-1, to=1-4]
	\arrow[""{name=1, anchor=center, inner sep=0}, "{(1,1)}"', curve={height=18pt}, from=1-1, to=1-4]
	\arrow["{(2,0)}"{description}, from=1-4, to=1-7]
	\arrow[""{name=2, anchor=center, inner sep=0}, "{(3,0)}"{description}, curve={height=-30pt}, from=1-7, to=1-10]
	\arrow[""{name=3, anchor=center, inner sep=0}, "{(3,2)}"{description}, curve={height=30pt}, from=1-7, to=1-10]
	\arrow[""{name=4, anchor=center, inner sep=0}, "{(3,1)}"{description}, from=1-7, to=1-10]
	\arrow[shorten <=5pt, shorten >=5pt, Rightarrow, from=0, to=1]
	\arrow[shorten <=4pt, shorten >=4pt, Rightarrow, from=2, to=4]
	\arrow[shorten <=4pt, shorten >=4pt, Rightarrow, from=4, to=3]
\end{tikzcd}\]

\vspace{2mm}

A morphism $[\alpha, \boldsymbol{\alpha}] : [m; \mathbf{p}] \to [n; \mathbf{q}]$ corresponds to the 2-functor which sends each object $i$ of the 2-category $[m; \mathbf{p}]$ to $\alpha(i)$, each 1-morphism $(i,k)$ to the composite $$(\alpha(i), \alpha_{\alpha(i)}(k)) \circ \ldots \circ (\alpha(i-1) + 1, \alpha_{\alpha(i-1) + 1}(k))$$ when $\alpha(i) > \alpha(i-1)$ and to $\id_{\alpha(i)}$ when $\alpha(i) = \alpha(i-1)$, and each 2-morphism of $[m; \mathbf{p}]$ to the unique possible 2-morphism of $[n; \mathbf{q}]$.
\end{parag}

\begin{parag} \label{2-cellular sets}
We consider the category $\widehat{\Theta_2}$ of \textit{2-cellular sets}. The representable 2-cellular sets are denoted by $\Theta_2[n; \mathbf{q}]$. The \textit{boundary} $\partial \Theta_2[n; \mathbf{q}]$ of a representable $\Theta_2[n; \mathbf{q}]$ is the 2-cellular set generated by the non-surjective\footnote{A morphism $[\alpha; \boldsymbol{\alpha}]: [m; \mathbf{p}] \to [n; \mathbf{q}]$ is surjective if $\alpha:[m] \to [n]$ is surjective and all $\alpha_j: [p_i] \to [q_j]$ are surjective.} morphisms $[m; \mathbf{p}] \to [n; \mathbf{q}]$ in $\Theta_2$. We denote by $\delta_{n; \mathbf{q}}$ the boundary inclusion $\partial \Theta_2[n; \mathbf{q}] \to \Theta_2[n; \mathbf{q}]$.
\end{parag}

%\begin{parag} \label{spine inclusions}
%Let us define the \textit{spines} $I[n; \mathbf{q}]$. We define $I[0]$, $I[1;0]$ and $I[1;1]$ to be $\Theta_2[0]$, $\Theta_2[1;0]$ and $\Theta_2[1;1]$ respectively. For $q \geq 2$, let $$I[1; q] := \colim (\Theta_2[1;1] \leftarrow \Theta_2[1;0] \to \ldots \leftarrow \Theta_2[1;0] \to \Theta_2[1;1])$$
%where there are $q$ copies of $\Theta_2[1;1]$, the left-going arrows are given by $[\id; \delta^0]$ and the right-going arrows are given by $[\id; \delta^1]$. For $n \geq 2$, let $$I[n; \mathbf{q}] := \colim (I[1;q_1] \leftarrow \Theta_2[0] \to \ldots \leftarrow \Theta_2[0] \to I[1;q_n])$$
%where the left-going arrows are given by $[\delta^0]$ and the right-going arrows are given by $[\delta^1]$. For every $[n; \mathbf{q}] \in \Theta_2$, there is a natural inclusion $i_{n; \mathbf{q}}: I[n; \mathbf{q}] \to \Theta_2[n; \mathbf{q}]$.
%\end{parag}

\begin{parag} \label{Ara's model structure}
There is a model category structure on $\widehat{\Theta_2}$, constructed in \cite{ara2014higher} by means of Cisinki's theory of localizers, which we will call \textit{Ara's model structure for 2-quasi-categories}. The cofibrations of this model structure are monomorphisms of 2-cellular sets (and so every object is cofibrant). The fibrant objects are called \textit{2-quasi-categories}.
\end{parag}

\begin{parag}\label{left Kan extension and nerve/singular functor}
Let $F: \a \to \b$ be a functor from a small category $\a$ to a category $\b$. The \textit{evaluation (or nerve) functor} associated to $F$ is the functor $F^!: \b \to \widehat{\a}$, given by
$F^!(b)_a = \Hom_{\b}(F(a), b)$ for every $b \in \b$ and every $a \in \a$.

If $\b$ is cocomplete, the evaluation functor admits a left adjoint $F_!: \widehat{\a} \to \b$, which is the left Kan extension of $F$ along the Yoneda embedding $\a \to \widehat{\a}$.
\end{parag}

\begin{parag} \label{the strict 2-nerve}
The inclusion $\Theta_2 \to \twocat$ induces a nerve functor $N_2: \twocat \to \widehat{\Theta_2}$, which we call the \textit{(strict) 2-nerve} functor. Explicitly, for a 2-category $\cc$ and $[n; \mathbf{q}] \in \Theta_2$, we have
$$N_2 (\cc)_{n; \mathbf{q}} = \Hom_{\twocat}([n; \mathbf{q}], \cc)$$

The strict 2-nerve functor is fully faithful, cf. \cite[Theorem 1.12]{berger2002cellular}.
\end{parag}

The (strict 1-)nerve of a category is always a quasi-category. The analogous statement is not true, in general, for 2-categories.

%\begin{proposition}[Recognition principle]
%Let $X$ be...
%\end{proposition}

\begin{proposition} \label{When the strict 2-nerve of a 2-category is a 2-quasi-category}
The strict 2-nerve of a 2-category $\cc$ is a 2-quasi-category if and only if the only invertible 2-morphisms of $\cc$ are the identities.
\end{proposition}

\begin{proof}
It follows from \cite[Proposition 7.10]{ara2014higher} in the case $n=2$.
\end{proof}

\section{2-quasi-groupoids}

In this section, we provide a definition of groupoidal 2-quasi-category, which we call 2-quasi-groupoid. We show that there is a model category structure on the category of 2-cellular sets such that the fibrant-cofibrant objects are precisely the 2-quasi-groupoids. Moreover, we show that this model category structure is Quillen equivalent to the Kan-Quillen model structure on simplicial sets.

\begin{parag} \label{Adjunction (t,i)}
The inclusion $i: \Cat \to \twocat$ has a left adjoint $t: \twocat \to \Cat$, called the \textit{truncation} functor, which sends a strict 2-category $\cc$ to the category $t(\cc)$ whose objects are the same as those of $\cc$ and morphisms are equivalence classes of 1-morphisms of $\cc$ with respect to the equivalence relation freely generated by 2-morphisms. The adjunction
$$t: \twocat \rightleftarrows \Cat : i$$
restricts to an adjunction 
$$t: \Theta_2 \rightleftarrows \Delta : i$$

Explicitly, we have $t([n; \mathbf{q}]) = [n]$ and $i([n]) = [n; 0 , \ldots, 0]$.

The adjunction above yields an adjunction
$$t^*: \widehat{\Delta} \rightleftarrows \widehat{\Theta_2} : i^*$$
between the respective presheaf categories. Note that the functor $t^*$ is isomorphic to the functor $i_!: \widehat{\Delta} \to \widehat{\Theta_2}$ which extends the functor $\Delta \xrightarrow[]{i} \Theta_2 \xrightarrow[]{} \widehat{\Theta_2}$ by colimits.
\end{parag}

\begin{proposition} \label{Quillen adjunction quasi-categories and 2-quasi-categories}
The adjunction $(t^*, i^*)$ is a Quillen adjunction between the category of simplicial sets with Joyal's model structure and the category of 2-cellular sets with Ara's model structure for 2-quasi-categories.
\end{proposition}

\begin{proof}
See \cite[Proposition 7.5]{campbell2020homotopy}.
\end{proof}

\begin{parag} \label{underlying quasi-category}
Let $X$ be a 2-quasi-category. Proposition \ref{Quillen adjunction quasi-categories and 2-quasi-categories} implies that the simplicial set $i^*(X)$ is a quasi-category. We call $i^*(X)$ the \textit{underlying quasi-category of $X$}.
\end{parag}

\begin{parag} \label{right truncation}
The inclusion $i: \Cat \to \twocat$ also admits a right adjoint $t_r: \twocat \to \Cat$, called the \textit{right truncation}. If $\cc$ is a 2-category, then $t_r(\cc)$ is the category whose objects and morphisms are the objects and (1-)morphisms of $\cc$. Using the adjunction $(i, t_r)$ and the definition of the nerve functors, it is easy to see that the square
\begin{center}
    \begin{tikzcd}
            \twocat \arrow[r, "t_r"] \arrow[d, "N_2" swap] & \Cat \arrow[d, "N"] \\
            \widehat{\Theta_2} \arrow[r, "i^*"] & \widehat{\Delta}
    \end{tikzcd}
\end{center}
is commutative (up to isomorphism). % Unlike $t$, the functor $t_r$ does not restrict to a functor $\Theta_2 \to \Delta$.
\end{parag}

\begin{parag} \label{Adjunction (Sigma, Hom)}
Let $X$ be a 2-cellular set, and $x,y \in X_0$. Consider the simplicial set $\Hom_X(x,y)$, whose set of $n$-simplices is given by the pullback
\begin{center}
    \begin{tikzcd}
            \Hom_X(x,y)_n \arrow[d] \arrow[r] & X_{1; n} \arrow[d] \\ % if not clear, copy and paste:  , "({\delta^1}^* \text{,}  {\delta^0}^*)"
            \{*\} \arrow[r, "(x \text{,} y)", swap] &  X_0 \times X_0
    \end{tikzcd}
\end{center}
for $n \geq 0$, and face and degeneracy maps induced by $X([\id; \delta^i])$ and $X([\id; \sigma^i])$, respectively.

The association $(X, x, y) \mapsto \Hom_X(x,y)$ is actually part of an adjunction between bipointed 2-cellular sets and simplicial sets: 
$$\Sigma: \widehat{\Delta} \rightleftarrows \partial \Theta_2[1;0] / \widehat{\Theta_2}: \Hom$$

The left adjoint $\Sigma$ is obtained by left Kan extension along the Yoneda embedding of  the functor $\Delta \to \partial \Theta_2[1;0] / \widehat{\Theta_2}$ which maps $[n]$ to $(\Theta_2[1; n], 0, 1)$ for every $n \geq 0$.
\end{parag}

\begin{proposition} \label{Quillen adjunction suspension Hom}
The adjunction $(\Sigma, \Hom)$ is a Quillen adjunction between the category of simplicial sets with Joyal's model structure and the category of bipointed 2-cellular sets with the model structure induced by Ara's model structure for 2-quasi-categories.
\end{proposition}

\begin{proof}
See \cite[Proposition 6.5]{campbell2020homotopy}.
\end{proof}

\begin{parag} \label{hom quasi-category}
Let $X$ be a 2-quasi-category. Proposition \ref{Quillen adjunction suspension Hom} implies that, for all $x,y \in X_0$ the simplicial set $\Hom_X(x,y)$ is a quasi-category. We call $\Hom_X(x,y)$ the \textit{hom-quasi-category} between $x$ and $y$. We say that $X$ is \textit{locally Kan} if, for every $x,y \in X_0$, the hom-quasi-category $\Hom_X(x,y)$ is a Kan complex.
\end{parag}

\begin{proposition} \label{locally Kan 2-quasi-categories as local objects}
A 2-quasi-category $X$ is locally Kan if and only if $X$ is local with respect to $[\id; \sigma^0]: \Theta_2[1;1] \to \Theta_2[1;0]$ in the model structure for 2-quasi-categories.
\end{proposition}

\begin{proof}
Let $\m$ be a model category and $C$ be a cofibrant object of $\m$. Let $(A, a: C \to A)$ and $(B, b: C \to B)$ be two objects of $C/\m$, and $f: A \to B$ be a morphism in $\m$ such that $f a = b$ (that is, $f$ is a morphism in $C/\m$). One can show that a fibrant object $X$ of $\m$ is local with respect to $f$ if and only if it for every morphism $x: C \to X$, the object $(X, x)$ of $C/\m$ is local with respect to $f$ in the induced model structure.

Taking $\m$ to be $\widehat{\Theta_2}$ with the model structure for 2-quasi-categories and $C$ to be $\partial \Theta_2[1;0]$, we get that a 2-quasi-category $X$ is local with respect to $[\id; \sigma^0]: \Theta_2[1;1] \to \Theta_2[1;0]$ if and only for every $x,y \in X_0$, the object $(X, x, y)$ of $\partial \Theta_2[1;0] / \widehat{\Theta_2}$ is local with respect $[\id; \sigma^0]$. But $[\id; \sigma^0] = \Sigma(\sigma^0: \Delta[1] \to \Delta[0])$. Thus, by applying Proposition \ref{local fibrant objects and localisation} to the adjunction $\Sigma: \widehat{\Delta} \rightleftarrows \partial \Theta_2[1;0]/ \widehat{\Theta_2}: \Hom$, we conclude that a 2-quasi-category $X$ is local with respect to $[\id; \sigma^0]$ if and only if for every $x,y \in X_0$, the hom-quasi-category $\Hom_X(x,y)$ is local with respect to $\Delta[1] \to \Delta[0]$, which is equivalent to saying that $\Hom_X(x,y)$ is a Kan complex (cf. Example \ref{Kan-Quillen vs Joyal}).
\end{proof}

Using the characterisation of locally Kan 2-quasi-categories as local objects in Ara's model structure, we can apply the existence theorem \ref{smith existence theorem} to get a \textit{model structure for locally Kan 2-quasi-categories}.

\begin{proposition} \label{existence of model structure for locally Kan 2-quasi-categories}
There exists a model structure on the category of 2-cellular sets whose cofibrations are monomorphisms, and whose fibrant objects are locally Kan 2-quasi-categories. This model structure is a Bousfield localisation of Ara's model structure with respect to $[\id; \sigma^0]: \Theta_2[1;1] \to \Theta_2[1;0]$.
\end{proposition}

\begin{proof}
Ara's model structure for 2-quasi-categories is left proper and combinatorial, so we can apply Theorem \ref{smith existence theorem} to $S = \{[\id; \sigma^0]: \Theta_2[1;1] \to \Theta_2[1;0]\}$.
\end{proof}

The following theorem of Campbell will be the starting point of the proof of Theorem \ref{Quillen equivalence (infty, 2)-groupoids and spaces}.

\begin{theorem} \label{Quillen equivalence quasi-categories and locally Kan 2-quasi-categories}
The adjunction $$t^*: \widehat{\Delta} \rightleftarrows \widehat{\Theta_2} : i^*$$ is a Quillen equivalence between the category of simplicial sets with Joyal's model structure and the category of 2-cellular sets with the model structure for locally Kan 2-quasi-categories.
\end{theorem}

\begin{proof}
See \cite[Corollary 11.6]{campbell2020homotopy}.
\end{proof}

We now have all the ingredients to define 2-quasi-groupoids. Recall that a (strict) 2-groupoid is a 2-category where all 1-morphisms and 2-morphisms are (strictly) invertible. Equivalently, a 2-category $\cc$ is a 2-groupoid if and only if:
\begin{enumerate}
    \item $\cc$ is \textit{locally groupoidal}, i.e., for every two objects $x, y$ of $\cc$, the hom-category $\cc(x,y)$ is a groupoid. Equivalently, for every two object $x, y$ of $\cc$, the quasi-category $N(\cc(x,y)) \cong \Hom_{N_2(\cc)}(x,y)$ is a Kan complex (or $\infty$-groupoid);
    \item $t_r(\cc)$ is a groupoid, which amounts to say that $N t_r (\cc) \cong i^* N_2(\cc)$ is a Kan complex.
\end{enumerate}

Inspired by this characterisation, we give the following definition:

\begin{definition} \label{(infty, 2)-groupoid}
A 2-quasi-category $X$ is a \textit{2-quasi-groupoid} if it is locally Kan and its underlying quasi-category $i^*(X)$ is a Kan complex.
\end{definition}

\begin{proposition} \label{(infty, 2)-groupoids as local objects}
A 2-quasi-category $X$ is a 2-quasi-groupoid if and only if $X$ is local with respect to $[\id; \sigma^0]: \Theta_2[1;1] \to \Theta_2[1;0]$ and $[\sigma^0]: \Theta_2[1;0] \to \Theta_2[0]$ in the model structure for 2-quasi-categories.
\end{proposition}

\begin{proof}
By Proposition \ref{locally Kan 2-quasi-categories as local objects}, it suffices to show that $i^*(X)$ is a Kan complex if and only if $X$ is local with respect to $[\sigma^0]: \Theta_2[1;0] \to \Theta_2[0]$. This follows from Proposition \ref{local fibrant objects and localisation} applied to the Quillen adjunction $t^*: \widehat{\Delta} \rightleftarrows \widehat{\Theta_2}: i^*$ of Proposition \ref{Quillen adjunction quasi-categories and 2-quasi-categories}, by noticing that $[\sigma^0]: \Theta_2[1;0] \to \Theta_2[0] = t^*(\Delta[1] \to \Delta[0])$ and that Kan complexes are precisely quasi-categories which are local with respect to $\Delta[1] \to \Delta[0]$ in Joyal's model structure.
\end{proof}

Exactly as in Proposition \ref{existence of model structure for locally Kan 2-quasi-categories}, we can get a \textit{model structure for 2-quasi-groupoids}.

\begin{proposition} \label{existence of model structure for (infty, 2)-groupoids}
There exists a model structure on the category of 2-cellular sets whose cofibrations are monomorphisms, and whose fibrant objects are 2-quasi-groupoids. This model structure is a Bousfield localisation of Ara's model structure with respect to $[\id; \sigma^0]: \Theta_2[1;1] \to \Theta_2[1;0]$ and $[\sigma^0]: \Theta_2[1;0] \to \Theta_2[0]$. \qed
\end{proposition}

We finish this section by showing that this model structure is Quillen equivalent to the Kan-Quillen model structure on simplicial sets.

\begin{theorem} \label{Quillen equivalence (infty, 2)-groupoids and spaces}
The adjunction $$t^*: \widehat{\Delta} \rightleftarrows \widehat{\Theta_2} : i^*$$ is a Quillen equivalence between the category of simplicial sets with the Kan-Quillen model structure and the category of 2-cellular sets with the model structure for 2-quasi-groupoids.
\end{theorem}

\begin{proof}
By theorem Theorem \ref{Quillen equivalence quasi-categories and locally Kan 2-quasi-categories}, the adjunction above is a Quillen equivalence between the model structures for quasi-categories and for locally Kan 2-quasi-categories. We know that the Kan-Quillen model structure is a localisation of Joyal's model structure with respect to $\Delta[1] \to \Delta[0]$. Applying Theorem \ref{transfer of localisation} to the adjunction in question, we obtain a Quillen equivalence between the Kan-Quillen model structure and the localisation of the model structure for locally Kan 2-quasi-categories with respect to $\L t^*(\Delta[1] \to \Delta[0]) = (\Theta_2[1;0] \to \Theta_2[0])$ (which exists by Theorem \ref{smith existence theorem} since the model structure for locally Kan 2-quasi-categories is left proper and combinatorial - also by Theorem \ref{smith existence theorem}).

It remains to show that this localisation is precisely the model structure for 2-quasi-groupoids. This follows immediately by Propositions \ref{proposition succesive localisations}, \ref{existence of model structure for locally Kan 2-quasi-categories} and \ref{existence of model structure for (infty, 2)-groupoids}.
\end{proof}

\section{Campbell's nerve for bicategories and 2-truncated 2-quasi-categories}

We have seen in Proposition \ref{When the strict 2-nerve of a 2-category is a 2-quasi-category} that, unlike the case of quasi-categories, the strict 2-nerve of a 2-category is not always a 2-quasi-category. In \cite{campbell2020homotopy} is built a homotopy coherent nerve $N_h: \Bicat \to \widehat{\Theta_2}$, with the property that $N_h(\cc)$ is a 2-quasi-category for every bicategory $\cc$. In this section, we recall the definition of the nerve $N_h$ and how it induces a Quillen equivalence between Lack's model structure for bicategories (described in \cite{lack2004quillen}) and a model structure for 2-truncated 2-quasi-categories (see Definition \ref{Definition of 2-truncated 2-quasi-categories}). 

All the results of this section are due to Campbell. We present some proofs in order to show that the techniques used in section \ref{section 2-truncated 2-quasi-groupoids vs 2-groupoids} are the same as those of \cite{campbell2020homotopy}.

\begin{parag} \label{bicategories and normal pseudofunctors}
We will write $\Bicat$ for the category of bicategories and normal pseudo-functors (i.e., pseudo-functors which preserve identities strictly and preserve composition of 1-morphisms up to an invertible 2-morphism) and $\Bicat_s$ for the the category of bicategories and strict functors. Recall that $\twocat$ is the category of (strict) 2-categories and (strict) 2-functors. We have inclusions:
$$\twocat \to \Bicat_s \to \Bicat$$ where only the first one is fully faithful. Both inclusion have left adjoints, which we denote by $\lambda: \Bicat_s \to \twocat$ and $Q: \Bicat \to \Bicat_s$. We write $\st$ for the composite functor $\st := \lambda Q: \Bicat \to \twocat$.

An important feature of the categories $\twocat$ and $\Bicat_s$ is that they are both complete and cocomplete, which is not the case for $\Bicat$ (it is nether complete nor cocomplete).
\end{parag}

\begin{definition} \label{homotopy coherent nerve}
The \textit{homotopy coherent nerve} $N_h: \Bicat \to \widehat{\Theta_2}$ is the nerve functor (cf. §\ref{left Kan extension and nerve/singular functor}) associated to the inclusion $\Theta_2 \to \twocat \to \Bicat$.
\end{definition}

As for the strict 2-nerve, the homotopy coherent nerve $N_h: \Bicat \to \widehat{\Theta_2}$ is fully faithful \cite[Theorem 3.18]{campbell2020homotopy}. We also denote by $N_h$ the restriction of this functor along the inclusions $\twocat \to \Bicat_s \to \Bicat$. Note that $N_h: \twocat \to \widehat{\Theta_2}$ is still faithful, but is no longer full.

Explicitly, for a 2-category $\cc$ and $[n; \mathbf{q}] \in \Theta_2$, we have 
$$N_h (\cc)_{n; \mathbf{q}} = \Hom_{\Bicat}([n; \mathbf{q}], \cc) \cong \Hom_{\Bicat_s}(Q[n; q], \cc) \cong  \Hom_{\twocat}(\st [n; \mathbf{q}], \cc)$$

\begin{parag} \label{definition of tau_b}
The restriction of the nerve $N_h: \Bicat \to \widehat{\Theta_2}$ to $\Bicat_s$ is isomorphic to the nerve functor $\Bicat_s \to \widehat{\Theta_2}$ induced by the composition of the inclusion $\Theta_2 \to \Bicat$ with the functor $Q: \Bicat \to \Bicat_s$. Since $\Bicat_s$ is a cocomplete category, we have (see §\ref{left Kan extension and nerve/singular functor}) that $N_h: \Bicat_s \to \widehat{\Theta_2}$ is the right functor of an adjunction
$$\tau_b: \widehat{\Theta_2} \rightleftarrows \Bicat_s: N_h$$
where $\tau_b$ is obtained by left Kan extension of $Q: \Theta_2 \to \Bicat_s$ along the Yoneda embedding.
%Composing the inclusion $\Theta_2 \to \Bicat$ with the functor $Q: \Bicat \to \Bicat_s$, we obtain a functor $Q: \Theta_2 \to \Bicat_s$, which produces via §\ref{left Kan extension and nerve/singular functor} and adjunction
%$$\tau_b: \widehat{\Theta_2} \rightleftarrows \Bicat_s: N_b$$

%Since $Q: \Bicat \to \Bicat_s$ is left adjoint to the inclusion $\Bicat_s \to \Bicat$, the functor $N_b: \Bicat_s \to \widehat{\Theta_2}$ is isomorphic to the homotopy coherent nerve $N_h: \Bicat_s \to \widehat{\Theta_2}$.
\end{parag}

\begin{parag}\label{Model category structures on 2-categories and bicategories}
A pseudo-functor $F: \cc \to \d$ between two bicategories is a \textit{biequivalence} if \begin{enumerate}
    \item $F$ is biessentially surjective on objects (i.e., if for every object $y$ of $\d$, there is an object $x$ of $\cc$ such that $F(x)$ is equivalent\footnote{A morphism $f: x \to y$ in a bicategory is an \textit{equivalence} if there exists a morphism $g: y \to x$ and invertible 2-morphisms $\alpha: gf \Rightarrow \id_x$ and $\beta: fg \Rightarrow \id_y$. Two objects of a bicategory are said to be \textit{equivalent} if there is an equivalence between then.} to $y$, and
        \item $F$ is locally an equivalence of categories (i.e., for every objects $x, x' \in \cc$, the functor $F_{x,x'}: \Hom_{\cc}(x,x') \to \Hom_{\d}(F x, F x')$ is an equivalence of categories)
\end{enumerate}

The categories $\Bicat_s$ and $\twocat$ are endowed with model category structures, described in \cite{lack2002quillen} and \cite{lack2004quillen}, whose weak equivalences are biequivalences and for which all objects are fibrant. The model structure on $\twocat$ is right-transferred by the one on $\Bicat_s$ via the adjunction $\lambda: \Bicat_s \rightleftarrows \twocat$ of §\ref{bicategories and normal pseudofunctors}. Moreover, this adjunction is a Quillen equivalence between both model categories. It will be useful to know hat the components at every object of the unit and the counit of the adjunction $\st: \Bicat \rightleftarrows \twocat$ are biequivalences (See \cite[§4.8]{campbell2020homotopy}). % It will be useful to know that the counit $Q \cc \to \cc$ of the adjunction $Q: \Bicat \rightleftarrows \Bicat_s$ provides a cofibrant replacement for every bicategory $\cc$ in $\Bicat_s$, and that the components at every object of the unit and the counit of the adjunction $\st: \Bicat \rightleftarrows \twocat$ are biequivalences (See \cite[§4.3-4.8]{campbell2020homotopy}).

\end{parag}

\begin{theorem}\label{Quillen adjunction 2-quasi-categories and bicategories}
The adjunction
$$\tau_b: \widehat{\Theta_2} \rightleftarrows \Bicat_s: N_h$$
is a Quillen adjunction between the category of 2-cellular sets with Ara's model structure for 2-quasi-categories and $\Bicat_s$ with Lacks's model structure for bicategories. Moreover, it is a homotopy reflection.
\end{theorem}

\begin{proof}
See \cite[Theorem 5.10]{campbell2020homotopy}.
\end{proof}

\begin{parag} \label{Homotopy coherent nerve as a fibrant replacement}
Theorem \ref{Quillen adjunction 2-quasi-categories and bicategories} implies that the homotopy coherent nerve $N_h(\cc)$ of every bicategory $\cc$ is a 2-quasi-category. Therefore, we have indeed a solution for the "problem" of the strict nerve described in the beginning of this section. Moreover, since every strict 2-functor is a normal pseudo-functor, there is an inclusion $N_2(\cc) \to N_h(\cc)$ for every 2-category $\cc$. Campbell shows that this inclusion is a weak equivalence of Ara \cite[Theorem 10.10]{campbell2020homotopy}, which exhibits $N_h(\cc)$ as a fibrant replacement of $N_2(\cc)$.
\end{parag}

\begin{parag} \label{Homotopy bicategory of a 2-quasi-category}
Since the homotopy coherent nerve of a bicategory is a 2-quasi-category, the functor $N_h: \Bicat \to \widehat{\Theta_2}$ factors through the full subcategory $\twoqcat$ of $\widehat{\Theta_2}$ formed by 2-quasi-categories. The functor $N_h: \Bicat \to \twoqcat$ has a left adjoint, denote by $\Ho: \twoqcat \to \Bicat$. For a 2-quasi-category $X$, we call $\Ho(X)$ the \textit{homotopy bicategory} associated to $X$. The objects of $\Ho(X)$ are the same as those of $X$, and its hom-categories are given by $\Ho(X)(x,y) = \ho(\Hom_X(x,y))$, for every $x,y \in X_0$ (where $\ho(Y)$ denotes the homotopy category of a quasi-category $Y$) \cite[§6.26]{campbell2020homotopy}. A 1-morphism of $\Ho(X)$ is an equivalence if and only if it is invertible in the underlying quasi-category $i^*(X)$ \cite[Proposition 7.8]{campbell2020homotopy}.

By considering the adjunctions of §\ref{bicategories and normal pseudofunctors} and §\ref{definition of tau_b}, we have that for every 2-quasi-category $X$ and every bicategory $\cc$, 
$$\Hom_{\Bicat_s}(Q \Ho(X), \cc) \cong \Hom_{\Bicat}(\Ho(X), \cc) \cong \Hom_{\widehat{\Theta_2}}(X, N_h(\cc)) \cong \Hom_{\Bicat_s}(\tau_b X, \cc)$$

By the Yoneda lemma, the bicategories $Q \Ho(X)$ and $\tau_b X$ are isomorphic. By applying the functor $\lambda$, we obtain an isomorphism $\st \Ho(X) \cong \lambda \tau_b X$.
\end{parag}

We can now give an alternative characterisation of 2-quasi-groupoids. Recall that a \textit{bigroupoid} is a bicategory where every 1-morphism is an equivalence and every 2-morphism is (strictly invertible).

\begin{proposition} \label{characterisation of (infty, 2)-groupoids as bigroupoids}
A 2-quasi-category $X$ is a 2-quasi-groupoid if and only if the 2-category $\lambda \tau_b (X)$ is a bigroupoid.
\end{proposition}

\begin{proof}
Let $X$ be a 2-quasi-category. We show that:
\begin{enumerate}
    \item $i^*(X)$ is a Kan complex if and only if every 1-morphism of $\lambda \tau_b (X)$ is an equivalence, and
    \item $X$ is locally Kan if and only if every 2-morphism of $\lambda \tau_b (X)$ is invertible.
\end{enumerate}

We have that $i^*(X)$ is a Kan complex if and only if every 1-morphism of $\Ho(X)$ is an equivalence, cf. §\ref{Homotopy bicategory of a 2-quasi-category}. We know that $\lambda \tau_b(X) \cong \st (\Ho(X))$ and that the unit $\Ho(X) \to \st(\Ho(X))$ is a biequivalence by §\ref{Model category structures on 2-categories and bicategories}. Note that if $\cc \to \d$ is a biequivalence between bicategories, then every morphism of $\cc$ is an equivalence if and only if every morphism of $\d$ is an equivalence. This shows (1).

For (2), recall that for every $x,y \in X_0$, we have $\ho(\Hom_X(x,y)) = \Ho(X)(x,y)$, so $X$ is locally Kan if and only if the homotopy bicategory $\Ho(X)$ is locally groupoidal. Since a bicategory biequivalent to a locally groupoidal bicategory is also locally groupoidal, we have that $\lambda \tau_b(X) \cong \st (\Ho(X))$ is locally groupoidal. This is equivalent to saying that every 2-morphism of $\lambda \tau_b (X)$ is invertible.
\end{proof}

\begin{parag}\label{Definition of 2-truncated 2-quasi-categories}
Recall from \cite{joyal2008notes, campbell2020truncated} that a quasi-category $X$ is said to be \textit{1-truncated} if the morphism $X \to N(\ho(X))$ is a weak categorical equivalence. Following \cite{campbell2020homotopy}, a 2-quasi-category $X$ is \textit{2-truncated} if, for every $x,y \in X_0$, the hom-quasi-category $\Hom_X(x,y)$ is 1-truncated.
\end{parag}

2-truncated 2-quasi-categories can be characterised as local objects.

\begin{proposition}\label{2-truncated 2-quasi-categories as local objects}
A 2-quasi-category $X$ is 2-truncated if and only if $X$ is local with respect to the boundary inclusion $\delta_{1;3}: \partial \Theta_2[1;3] \to \Theta_2[1;3]$ in the model structure for 2-quasi-categories.
\end{proposition}

\begin{proof}
A quasi-category is 1-truncated if and only if it is local with respect to the boundary inclusion $\delta_3: \partial \Delta[3] \to \Delta[3]$ in Joyal's model structure (see \cite[Proposition 3.23]{campbell2020truncated}). The result follows from Proposition \ref{local fibrant objects and localisation}, since $(\delta_{1;3}: \partial \Theta_2[1;3] \to \Theta_2[1;3]) = \Sigma(\delta_3: \partial \Delta[3] \to \Delta[3])$.
\end{proof}

\begin{parag} \label{model structure for 2-truncated 2-quasi-categories}
In view of the previous proposition, we can use Theorem \ref{smith existence theorem} once again to produce a Bousfield localisation of Ara's model structure, whose fibrant objects are the 2-truncated 2-quasi-categories. We call it the \textit{model structure for 2-truncated 2-quasi-categories}.
\end{parag}

The following theorem allows us to show that the Quillen adjunction of Theorem \ref{Quillen adjunction 2-quasi-categories and bicategories} becomes a Quillen equivalence after localizing Ara's model structure to obtain the model structure for 2-truncated 2-quasi-categories.

\begin{theorem} \label{2-truncated 2-quasi-categories are nerves of bicategories}
A 2-quasi-category $X$ is 2-truncated if and only if the unit $X \to N_h(\Ho(X))$ of the adjunction $\Ho: \twoqcat \rightleftarrows \Bicat: N_h$ is a weak equivalence of Ara.
\end{theorem}

\begin{proof}
See \cite[Theorem 7.28]{campbell2020homotopy}.
\end{proof}

\begin{theorem}\label{Quillen equivalence 2-truncated 2-quasi-categories and 2-categories}
The adjunction
$$\tau_b: \widehat{\Theta_2} \rightleftarrows \Bicat_s: N_h$$
is a Quillen equivalence between the category of 2-cellular sets with the model structure for 2-truncated 2-quasi-categories and $\Bicat_s$ with Lacks's model structure for bicategories.
\end{theorem}

\begin{proof}
Given Theorem \ref{Quillen adjunction 2-quasi-categories and bicategories}, it is sufficient by Theorem \ref{criterion homotopy refection Quillen equivalence} to show that every 2-truncated 2-quasi-category is weakly equivalent (in Ara's model structure) to the homotopy coherent nerve of a bicategory. This follows from Theorem \ref{2-truncated 2-quasi-categories are nerves of bicategories}.
\end{proof}

\section{2-truncated 2-quasi-groupoids vs 2-groupoids} \label{section 2-truncated 2-quasi-groupoids vs 2-groupoids}

In this section, we construct a Quillen equivalence between a model structure for 2-truncated 2-quasi-groupoids on $\widehat{\Theta_2}$ and the model structure for 2-groupoids on the category $\2Gpd$ of 2-groupoids and (strict) 2-functors, described in \cite{moerdijk1993algebraic}.

\begin{parag} \label{left adjoint of inclusion of groupoids}
The inclusion functor $\2Gpd \to \twocat$ has a left adjoint $F: \twocat \to \2Gpd$, which freely adds inverses to every morphism and 2-morphism.
\end{parag}

\begin{proposition} \label{adjunction 2-categories and 2-groupoids}
The adjunction 
$$F: \twocat \rightleftarrows \2Gpd$$
is a Quillen adjunction between Lack's model structure for 2-categories and Moerdijk-Svensson's model structure for 2-groupoids. The model structure for 2-groupoids is the model structure right-transferred from the model structure for 2-categories via this adjunction. Moreover, the adjunction is a homotopy reflection.
\end{proposition}

\begin{proof}
The fact that the model structure for 2-groupoids is the model structure right-transferred from the model structure for 2-categories can be checked directly from the definitions of weak equivalences and fibrations in both model structures. This implies that the adjunction is Quillen, since the inclusion $\2Gpd \to \twocat$ obviously preserves fibrations and weak equivalence. The adjunction is a homotopy reflection by a result of Lack \cite[Theorem 8.4]{lack2002quillen}, which states that the counit of the derived adjunction is invertible.
\end{proof}

\begin{parag} \label{recollection adjunctions}
Let us recollect, for the convenience of the reader, the adjunctions considered so far.

% https://q.uiver.app/?q=WzAsNixbMCwwLCJcXG1hdGhiZntCaWNhdH0iXSxbMiwwLCJcXG1hdGhiZntCaWNhdH1fcyJdLFs0LDAsIlxcbWF0aGJmezJcXHRleHR7LX1DYXR9Il0sWzYsMCwiXFxtYXRoYmZ7MlxcdGV4dHstfUdwZH0iXSxbMiwzLCJcXHdpZGVoYXR7XFxUaGV0YV8yfSJdLFswLDMsIlxcbWF0aGJmezJcXHRleHR7LX1RQ2F0fSJdLFsxLDQsIk5faCIsMCx7Im9mZnNldCI6LTJ9XSxbNCwxLCJcXHRhdV9iIiwwLHsib2Zmc2V0IjotMn1dLFsxLDAsIiIsMix7Im9mZnNldCI6LTJ9XSxbMiwxLCIiLDIseyJvZmZzZXQiOi0yfV0sWzMsMiwiSSIsMCx7Im9mZnNldCI6LTJ9XSxbMCwxLCJRIiwwLHsib2Zmc2V0IjotMn1dLFsxLDIsIkMiLDAseyJvZmZzZXQiOi0yfV0sWzIsMywiRiIsMCx7Im9mZnNldCI6LTJ9XSxbMCw1LCJOX2giLDAseyJvZmZzZXQiOi0yfV0sWzUsMCwiXFx0ZXh0e0hvfSIsMCx7Im9mZnNldCI6LTJ9XSxbNSw0XSxbNyw2LCIiLDAseyJsZXZlbCI6MSwic3R5bGUiOnsibmFtZSI6ImFkanVuY3Rpb24ifX1dLFsxMSw4LCIiLDAseyJsZXZlbCI6MSwic3R5bGUiOnsibmFtZSI6ImFkanVuY3Rpb24ifX1dLFsxMiw5LCIiLDAseyJsZXZlbCI6MSwic3R5bGUiOnsibmFtZSI6ImFkanVuY3Rpb24ifX1dLFsxMywxMCwiIiwwLHsibGV2ZWwiOjEsInN0eWxlIjp7Im5hbWUiOiJhZGp1bmN0aW9uIn19XSxbMTUsMTQsIiIsMCx7ImxldmVsIjoxLCJzdHlsZSI6eyJuYW1lIjoiYWRqdW5jdGlvbiJ9fV1d
\[\begin{tikzcd}[ampersand replacement=\&]
	{\mathbf{Bicat}} \&\& {\mathbf{Bicat}_s} \&\& {\mathbf{2\text{-}Cat}} \&\& {\mathbf{2\text{-}Gpd}} \\
	\\
	\\
	{\mathbf{2\text{-}QCat}} \&\& {\widehat{\Theta_2}}
	\arrow[""{name=0, anchor=center, inner sep=0}, "{N_h}", shift left=2, from=1-3, to=4-3]
	\arrow[""{name=1, anchor=center, inner sep=0}, "{\tau_b}", shift left=2, from=4-3, to=1-3]
	\arrow[""{name=2, anchor=center, inner sep=0}, shift left=2, from=1-3, to=1-1]
	\arrow[""{name=3, anchor=center, inner sep=0}, shift left=2, from=1-5, to=1-3]
	\arrow[""{name=4, anchor=center, inner sep=0}, "I", shift left=2, from=1-7, to=1-5]
	\arrow[""{name=5, anchor=center, inner sep=0}, "Q", shift left=2, from=1-1, to=1-3]
	\arrow[""{name=6, anchor=center, inner sep=0}, "\lambda", shift left=2, from=1-3, to=1-5]
	\arrow[""{name=7, anchor=center, inner sep=0}, "F", shift left=2, from=1-5, to=1-7]
	\arrow[""{name=8, anchor=center, inner sep=0}, "{N_h}", shift left=2, from=1-1, to=4-1]
	\arrow[""{name=9, anchor=center, inner sep=0}, "{\text{Ho}}", shift left=2, from=4-1, to=1-1]
	\arrow[from=4-1, to=4-3]
	\arrow["\dashv"{anchor=center}, draw=none, from=1, to=0]
	\arrow["\dashv"{anchor=center, rotate=-90}, draw=none, from=5, to=2]
	\arrow["\dashv"{anchor=center, rotate=-90}, draw=none, from=6, to=3]
	\arrow["\dashv"{anchor=center, rotate=-90}, draw=none, from=7, to=4]
	\arrow["\dashv"{anchor=center}, draw=none, from=9, to=8]
\end{tikzcd}\]
\end{parag}

In the diagram above, the non-labelled arrows are inclusions, the first two horizontal adjunctions are those of §\ref{bicategories and normal pseudofunctors} and the third horizontal adjunction is the one of Proposition \ref{adjunction 2-categories and 2-groupoids}. The left-side (resp. right-side) vertical adjunction is presented in §\ref{Homotopy bicategory of a 2-quasi-category} (resp. §\ref{definition of tau_b}). 

\begin{parag} \label{main Quillen adjunction}
We consider the Quillen adjunction
$$\tau: \widehat{\Theta_2} \rightleftarrows \2Gpd: N_h$$
obtained by composing the right-side vertical adjunction with the second and the third horizontal adjunctions of the diagram above. The functor $\tau: \widehat{\Theta_2} \to \2Gpd $ is explicitly defined as the composite $\tau:= F \lambda \tau_b$.
\end{parag}

Our goal is to show that this adjunction is a Quillen equivalence between a certain model structure for 2-truncated 2-quasi-groupoids on $\widehat{\Theta_2}$, described in §\ref{model structure for 2-truncated (infty, 2)-groupoids}, and Moerdijk-Svensson's model structure on $\2Gpd$.

\begin{proposition} \label{2-truncated (infty, 2)-groupoids as local objects}
A 2-quasi-category $X$ is a 2-truncated 2-quasi-groupoid if and only if it is local with respect to the morphisms
$$\delta_{1;3}: \partial \Theta_2[1;3] \to \Theta_2[1;3], [\id; \sigma^0]: \Theta_2[1;1] \to \Theta_2[1;0], [\sigma^0]: \Theta_2[1;0] \to \Theta_2[0]$$
in Ara's model structure for 2-quasi-categories
\end{proposition}

\begin{proof}
It follows from Propositions \ref{(infty, 2)-groupoids as local objects} and \ref{2-truncated 2-quasi-categories as local objects}.
\end{proof}

\begin{parag} \label{model structure for 2-truncated (infty, 2)-groupoids}
Theorem \ref{smith existence theorem} gives a model structure on $\widehat{\Theta_2}$ whose cofibrations are monomorphisms, and whose fibrant objects are 2-truncated 2-quasi-groupoids. This \textit{model structure for 2-truncated 2-quasi-groupoids} is the localisation of Ara's model structure for 2-quasi-categories with respect to the morphisms of Proposition \ref{2-truncated (infty, 2)-groupoids as local objects}.
\end{parag}

\begin{theorem} \label{Quillen equivalence 2-truncated (infty, 2)-groupoids and 2-groupoids}
The adjunction
$$\tau: \widehat{\Theta_2} \rightleftarrows \2Gpd: N_h$$
of §\ref{main Quillen adjunction} is a Quillen equivalence between the model structure for 2-truncated 2-quasi-groupoids and Moerdijk-Svensson's model structure for 2-groupoids.
\end{theorem}

\begin{proof}
We consider the model structure for 2-truncated quasi-categories on $\widehat{\Theta_2}$. With this model structure, the adjunction (1) $\tau: \widehat{\Theta_2} \rightleftarrows \2Gpd: N_h$ is Quillen and a homotopy reflection, as a composition of the Quillen equivalence (2) $\tau_b: \widehat{\Theta_2} \rightleftarrows \Bicat_s: N_h$ of Theorem \ref{Quillen equivalence 2-truncated 2-quasi-categories and 2-categories}, the Quillen equivalence (3) $ \lambda: \Bicat_s \rightleftarrows \twocat$ of §\ref{Model category structures on 2-categories and bicategories} and the homotopy reflection (4) $F: \twocat \rightleftarrows \2Gpd$ of Proposition \ref{adjunction 2-categories and 2-groupoids}.

By Propositions \ref{proposition succesive localisations} and \ref{model structure for 2-truncated 2-quasi-categories} and §\ref{model structure for 2-truncated (infty, 2)-groupoids}, the model structure for 2-truncated 2-quasi-groupoids is a localisation of the former one. We can then use Theorem \ref{criterion homotopy refection Quillen equivalence}.(2) to show that the adjunction (1) is a Quillen equivalence after localisation. The conditions to be checked amount to the assertions (a) and (b) below.

(a) The adjunction (1) is a Quillen adjunction, when considering the model structure for 2-truncated quasi-groupoids on $\widehat{\Theta_2}$. 

By Theorem \ref{criterion homotopy refection Quillen equivalence}.(1), it is sufficient to show that if $\g$ is a 2-groupoid, then $N_h(\g)$ is a 2-truncated 2-quasi-groupoid. For every objects $x, y$ of $\g$, we have that $\Hom_{N_h(\g)}(x,y) \cong N(\g(x,y))$ (here, the hom-simplicial set is the one defined in §\ref{hom quasi-category}). Indeed, an $n$-simplex of $\Hom_{N_h(\g)}(x,y)$ is a normal pseudo-functor $[1;n] \to \g$ sending the objects 0 and 1 of $[1;n]$ to $x$ and $y$ -- such a pseudo-functor is necessarily a 2-functor since $[1;n]$ has no composable (non-identity) 1-morphisms -- or, equivalently, the data of $n$ composable 2-morphisms with 2-source $x$ and 2-target $y$, which is an $n$-simplex of $N(\g(x,y))$. But $\g(x,y)$ is a groupoid, so $N(\g(x,y))$ is 1-truncated and a Kan complex, and hence $N_h(\g)$ is 2-truncated and locally Kan. Since every 1-morphism of $\g$ is invertible, it follows from the definitions that every morphism (1-simplex) of $i^*(N_h(\g))$ is invertible, and so $i^*(N_h(\g))$ is a Kan complex.

(b) For every 2-truncated 2-quasi-groupoid $X$, there exists a 2-groupoid $Y$ and a weak equivalence $X \to N_h(Y)$ in the model structure for 2-truncated 2-quasi-categories.

Let $Y = \tau X$ and consider the unit morphism $X \to N_h (\tau X)$ of the adjunction (1). This morphism can be written as a composite
$$X \xrightarrow[]{\eta_X} N_h(\lambda\tau_b X) \xrightarrow[]{N_h(\eta'_{\lambda \tau_b X})} N_h(\tau X)$$
where the first morphism is the component at $X$ of the unit of the Quillen equivalence obtained by composing the Quillen equivalences (2) and (3), and the second one is the image by $N_h$ of the component at $\lambda \tau_b X$ of the unit of the Quillen adjunction (4). The first morphism is a weak equivalence, since it is the unit of a Quillen equivalence at a cofibrant object $X$, and $\lambda \tau_b X$ is fibrant (as every 2-category). By Proposition \ref{characterisation of (infty, 2)-groupoids as bigroupoids}, the 2-category $\lambda \tau_b X$ is a bigroupoid. In the proof of \cite[Theorem 8.5]{lack2002quillen}, it is shown that if a 2-category $\a$ is a bigroupoid, the 2-functor $\a \to F \a$ given by the unit of the adjunction (4) is a biequivalence. Thus, the second morphism is also a weak equivalence, as the image by a right Quillen functor of a weak equivalence between fibrant objects. We conclude that $X \to N_h (\tau X)$ is a weak equivalence.
\end{proof}

\section{Homotopy 2-types}

In this section, we recall a result of \cite{moerdijk1993algebraic}, which will allow the comparison of 2-truncated 2-quasi-groupoids and homotopy 2-types, by means of a zigzag of Quillen equivalences.

\begin{parag} \label{review of homotopy types}
Let $n \geq 0$. A \textit{(homotopy) $n$-type} is a Kan complex $X$ such that, for every $x \in X_0$, the $m$-th homotopy group $\pi_m(X, x)$ is trivial for every $m > n$. Homotopy $n$-types can be characterised as local objects in the Kan-Quillen model structure on simplicial sets. Indeed, a Kan complex $X$ is an $n$-type if and only if it is local with respect to the boundary inclusion $\partial \Delta[n+2] \to \Delta[n+2]$ in this model structure (see for example \cite[Corollary 3.25]{campbell2020truncated}). Therefore, by Theorem \ref{smith existence theorem}, the Bousfield localisation of the Kan-Quillen model structure with respect to $\delta_n: \partial \Delta[n+2] \to \Delta[n+2]$ produces a model structure whose fibrant objects are homotopy $n$-types. For $n=2$, we obtain a \textit{model structure for homotopy 2-types}.
\end{parag}

\begin{theorem} [Moerdijk-Svensson] \label{Quillen equivalence 2-types and 2-groupoids}
There is a Quillen equivalence
$$W: \widehat{\Delta} \rightleftarrows \2Gpd: N_S$$
between the model structure for homotopy 2-types and Moerdijk-Svensson's model structure for 2-groupoids.
\end{theorem}

\begin{proof}
By \cite[Proposition 2.1.(ii),(iv)]{moerdijk1993algebraic}, the adjunction of \cite[Theorem 2.3]{moerdijk1993algebraic} is Quillen when considering the Kan-Quillen model structure on $\widehat{\Delta}$. It follows from Theorem \ref{criterion homotopy refection Quillen equivalence}.(1) and \cite[Proposition 2.1.(iii)]{moerdijk1993algebraic} that it remains Quillen after localisation to the model structure for 2-types. The induced adjunction between homotopy categories is an equivalence by \cite[Corollary 2.6]{moerdijk1993algebraic}, after noticing  that the homotopy category of 2-types is equivalent to the full subcategory of the homotopy category of spaces formed by homotopy 2-types.
\end{proof}

%\begin{remark} \label{statement of equivalences}
%Theorem \ref{Quillen equivalence 2-types and 2-groupoids} is not stated exactly as above in \cite{moerdijk1993algebraic}. 
%\end{remark}

%\begin{proof}
%The adjunction $(W, N_{S})$ is Quillen when considering the Kan-Quillen model structure. This comes from the fact the $N_{S}$ sends fibrations and weak equivalences of 2-groupoids to Kan fibrations and homotopy equivalences, respectively \cite[Proposition 2.1.(ii),(iv)]{moerdijk1993algebraic}. 

%In particular, $N_{S}$ sends a 2-groupoid to a Kan complex, which is a 2-type by \cite[Proposition 2.1.(iii)]{moerdijk1993algebraic}. So the nerve functor satisfies the condition (1) of Theorem \ref{criterion homotopy refection Quillen equivalence}, which implies that the adjunction remains a Quillen adjunction after localisation to the model structure for homotopy 2-types.

%Let $\widehat{\Delta}_{KQ}$ denote the category of simplicial sets with the Kan-Quillen model structure, and $\widehat{\Delta}_2$ denote the category of simplicial sets with the model structure for homotopy 2-types. Since $\widehat{\Delta}_2$ is a Bousfield localisation of $\widehat{\Delta}_{KQ}$, the right-derived identity functor $\Ho(\widehat{\Delta}_2) \to \Ho(\widehat{\Delta}_{KQ})$ is fully faithful, so $\Ho(\widehat{\Delta}_2)$ is equivalent to the full subcategory of $\Ho(\widehat{\Delta}_{KQ})$ formed by homotopy 2-types. Therefore, by \cite[Corollary 2.6]{moerdijk1993algebraic}, the right-derived functor $\R N_{S}: \Ho(\2Gpd) \to \Ho(\widehat{\Delta}_2)$ is an equivalence of categories.
%\end{proof}

\begin{corollary} \label{zigzag of Quillen equivalences}
The homotopy category of $\widehat{\Theta_2}$ with the model structure for 2-truncated 2-quasi-groupoids is equivalent to the homotopy category of $\widehat{\Delta}$ with the model structure for homotopy 2-types.
\end{corollary}

\begin{proof}
This follows from the Quillen equivalence of Theorem \ref{Quillen equivalence 2-truncated (infty, 2)-groupoids and 2-groupoids} and Moerdijk-Svensson's Quillen equivalence of Theorem \ref{Quillen equivalence 2-types and 2-groupoids}. Indeed, we have a zigzag of Quillen equivalences
% https://q.uiver.app/?q=WzAsMyxbNCwwLCJcXHdpZGVoYXR7XFxEZWx0YX0iXSxbMiwwLCJcXG1hdGhiZnsyXFx0ZXh0ey19R3BkfSJdLFswLDAsIlxcd2lkZWhhdHtcXFRoZXRhXzJ9Il0sWzAsMSwiVyIsMix7Im9mZnNldCI6Mn1dLFsxLDAsIk5fUyIsMix7Im9mZnNldCI6Mn1dLFsyLDEsIlxcdGF1IiwwLHsib2Zmc2V0IjotMn1dLFsxLDIsIk5faCIsMCx7Im9mZnNldCI6LTJ9XSxbMyw0LCIiLDAseyJsZXZlbCI6MSwic3R5bGUiOnsibmFtZSI6ImFkanVuY3Rpb24ifX1dLFs1LDYsIiIsMCx7ImxldmVsIjoxLCJzdHlsZSI6eyJuYW1lIjoiYWRqdW5jdGlvbiJ9fV1d
\[\begin{tikzcd}
	{\widehat{\Theta_2}} && {\mathbf{2\text{-}Gpd}} && {\widehat{\Delta}}
	\arrow[""{name=0, anchor=center, inner sep=0}, "W"', shift right=2, from=1-5, to=1-3]
	\arrow[""{name=1, anchor=center, inner sep=0}, "{N_S}"', shift right=2, from=1-3, to=1-5]
	\arrow[""{name=2, anchor=center, inner sep=0}, "\tau", shift left=2, from=1-1, to=1-3]
	\arrow[""{name=3, anchor=center, inner sep=0}, "{N_h}", shift left=2, from=1-3, to=1-1]
	\arrow["\dashv"{anchor=center, rotate=-90}, draw=none, from=0, to=1]
	\arrow["\dashv"{anchor=center, rotate=-90}, draw=none, from=2, to=3]
\end{tikzcd}\]
\end{proof}

\bibliographystyle{alpha}
\bibliography{sample}

\end{document}